\documentclass[12pt,a4paper]{amsart}
\usepackage{amsmath, amssymb, amsthm, yhmath}
\input cyracc.def

\usepackage{graphicx}
\usepackage{caption}  
\usepackage{float}    
\usepackage{mathtools} 
\usepackage[colorlinks=false,
            linkbordercolor={0 0 1},   
            citebordercolor={0 1 0},   
            urlbordercolor={1 0 0}     
]{hyperref}
\numberwithin{equation}{section}
\usepackage{enumerate}
\newtheorem{theorem}{Theorem}[section]
\newtheorem{lemma}[theorem]{Lemma}

\newtheorem{corollary}[theorem]{Corollary}
\newtheorem{proposition}[theorem]{Proposition}

\theoremstyle{definition} 

\newtheoremstyle{remarkstyle} 
   {}   
   {}   
  {\normalfont}  
  {}      
  {\itshape}  
  {.}     
  { }     
  {}      
\theoremstyle{remarkstyle}
\newtheorem*{remark}{Remark} 

\allowdisplaybreaks[4]

\begin{document}

\author[K. Matsuzaki]{Katsuhiko Matsuzaki \textsuperscript{$a$}} 
\address{$a.$ Department of Mathematics, School of Education, Waseda University, Tokyo 169-8050, Japan} 
\email{matsuzak@waseda.jp} 

\author[F. Tao]{Fei Tao \textsuperscript{$b$}} 
\address{$b.$ Beijing International Center for Mathematical Research, Peking University, Beijing 100871, P. R. China} 
\email{ferrytau@pku.edu.cn}

\thanks{The first author is partially supported by Japan Society for the Promotion of Science (KAKENHI 23K25775 and 23K17656); the second author is partially supported by the National Key R \& D Program of China (2025YFA1017500)}


\title{Graph Weldings Associated with Functions in Zygmund, $\rm{BMO}$, $\mathrm{VMO}$, and $H^{1/2}$}
\begin{abstract}
Let $f\colon \mathbb{R}\to\mathbb{R}$ be a continuous function. 
We define the graph welding associated with $f$ as the homeomorphism 
\[
\varphi = G^{-1}\circ F\colon \mathbb{R}\to\mathbb{R}.
\]
Here, $F(x)=x+if(x)$ parametrizes the graph of $f$, and $G$ is a conformal mapping from the upper half-plane $\mathbb{H}$ onto one of the two domains bounded by the graph of $f$, admitting a continuous extension to $\mathbb{R}$.
In this paper, we investigate how the regularity of the graph function $f$ influences the analytic properties of the associated graph welding $\varphi$. 
In particular, under certain assumptions that $f$ within Zygmund, $\mathrm{BMO}$, $\mathrm{VMO}$, and Hardy spaces, we establish results on absolute continuity, quasisymmetry, symmetry, strong quasisymmetry, strong symmetry, and the Weil--Petersson property of the associated graph welding $\varphi$. 
These results clarify the interplay between the regularity of graph functions, the geometry of graph curves, and the boundary behavior of conformal mappings.
\end{abstract}
\maketitle

\section{Introduction}

Let $\Gamma\subset\widehat{\mathbb C}=\mathbb{C}\cup\{\infty\}$ be a Jordan curve passing through $\infty$, and let $\Omega_+$ and $\Omega_-$ denote the complementary components of $\Gamma$. 
By the Riemann mapping theorem, there exist conformal mappings
\[
G_+\colon \mathbb H\to\Omega_+,
\qquad
G_-\colon \mathbb H^*\to\Omega_-
\]
such that $G_\pm(\infty)=\infty$, where $\mathbb H$ and $\mathbb H^*$ denote the upper and lower half-planes, respectively. 
By the Carath\'eodory theorem, the boundary extensions of $G_+$ and $G_-$ determine a homeomorphism
\[
h=G_-^{-1}\circ G_+|_{\mathbb{R}}\colon \mathbb{R}\to\mathbb{R},
\]
called the \emph{conformal welding} associated with $\Gamma$.

A fundamental theorem of Beurling and Ahlfors \cite{BA} states that a homeomorphism $h\colon\mathbb{R}\to\mathbb{R}$ arises as the conformal welding of a quasicircle if and only if $h$ is quasisymmetric (see below for the definition). 

Under the identification of the universal Teichm\"uller space with the quotient of the group of quasisymmetric homeomorphisms of $\mathbb{R}$ by $\mathrm{Aff}(\mathbb{R})$, the subgroup of all real affine mappings $z\mapsto az + b$, $a > 0$, $b\in\mathbb{R}$ (see, for example, \cite{Le}), many analytic subclasses of these homeomorphisms correspond naturally to both finer geometric properties of the associated curves and distinguished subspaces of the universal Teichm\"uller space.

Below, we summarize several important classes of quasisymmetric homeomorphisms, together with the corresponding curve classes and Teichm\"uller spaces.

\begin{itemize}
   \item Quasisymmetric class $\mathrm{QS}(\mathbb{R})$. An increasing homeomorphism $h\colon \mathbb{R}\to \mathbb{R}$ is \emph{quasisymmetric} if there exists a constant $M\ge1$ such that
      \[
      \frac1M\le\frac{h(x+t)-h(x)}{h(x)-h(x-t)}\leq M
      \]
      for all $x\in\mathbb{R}$ and $t>0$. The collection of all quasisymmetric homeomorphisms is denoted by $\mathrm{QS}(\mathbb{R})$.
   \begin{itemize}
      \item Curve geometry: quasicircles.
      \item Teichm\"uller space: the universal Teichm\"uller space
      \[
      T=\mathrm{Aff}(\mathbb{R})\backslash \mathrm{QS}(\mathbb{R}).
      \]
   \end{itemize}
      \item Symmetric class $\mathrm{S}(\mathbb{R})$. A quasisymmetric homeomorphism $h\colon \mathbb{R}\to \mathbb{R}$ is called \emph{symmetric} if
   \[
   \lim_{t\to0}\frac{h(x+t)-h(x)}{h(x)-h(x-t)}=1
   \]
   uniformly for $x\in\mathbb{R}$. The collection of all symmetric homeomorphisms is denoted by $\mathrm{S}(\mathbb{R})$. 
   This class was first studied in \cite{Ca}, in which Carleson investigated the absolute continuity of quasisymmetric homeomorphisms. 
   Subsequently, it was studied in depth by Gardiner and Sullivan \cite{GS} in connection with the little universal Teichm\"uller space.
   \begin{itemize}
      \item Curve geometry: asymptotically conformal quasicircles. This correspondence was established by Pommerenke \cite[Theorem 1]{Po78} in the unit-circle setting, and later extended to the real-line setting in the relative sense by Matsuzaki and Tao \cite[Theorem 1.1]{MT1}.
      \item Teichm\"uller space: the little universal Teichm\"uller space
         \[
         T_0=\mathrm{Aff}(\mathbb{R})\backslash \mathrm{S}(\mathbb{R}).
         \]
   \end{itemize}
   \item Strongly quasisymmetric class $\mathrm{SQS}(\mathbb{R})$. 
   A quasisymmetric homeomorphism \\$h\colon \mathbb{R}\to \mathbb{R}$ is \emph{strongly quasisymmetric} if it is locally absolutely continuous and $h'$ is a Muckenhoupt $A_\infty$-weight. 
   The collection of all strongly quasisymmetric homeomorphisms is denoted by $\mathrm{SQS}(\mathbb{R})$.
   \begin{itemize}
      \item Curve geometry: Bishop--Jones quasicircles. 
      The geometric characterization of this class is due to Bishop and Jones \cite[Theorem 4]{BJ} via the harmonic measures. 
      Furthermore, these quasicircles play an important role in the study of chord-arc curves; see, for instance, \cite{AG16,Sem88}.
      \item Teichm\"uller space: the $\mathrm{BMOA}$-Teichm\"uller space
         \[
         T_b=\mathrm{Aff}(\mathbb{R})\backslash \mathrm{SQS}(\mathbb{R}),
         \]
         which was introduced by Astala and Zinsmeister in \cite{AZ}. 
         This space naturally bridges conformal geometry and harmonic analysis, as its elements are characterized by Beltrami coefficients that induce Carleson measures.
   \end{itemize}
   \item Strongly symmetric class $\mathrm{SS}(\mathbb{R})$. 
   A strongly quasisymmetric homeomorphism $h\colon \mathbb{R}\to \mathbb{R}$ is \emph{strongly symmetric} if $\log h'\in \mathrm{VMO}(\mathbb{R})$, where $\mathrm{VMO}(\mathbb{R})$ denotes the space of functions of vanishing mean oscillation.
   The collection of all strongly symmetric homeomorphisms is denoted by $\mathrm{SS}(\mathbb{R})$.
   \begin{itemize}
      \item Curve geometry: asymptotically smooth quasicircles. 
      Shen and Wei \cite[Theorem 4.1]{SW} proved this correspondence for the unit circle, which Matsuzaki and Tao \cite[Theorem 1.2]{MT1} subsequently adapted to the real-line setting in the relative sense. 
      We see that $\mathrm{SS}(\mathbb{R})$ is contained in both $\mathrm{SQS}(\mathbb{R})$ and $\mathrm{S}(\mathbb{R})$. 
      However, the inclusion of $\mathrm{SS}(\mathbb{R})$ in $\mathrm{SQS}(\mathbb{R}) \cap \mathrm{S}(\mathbb{R})$ is strict as shown in \cite{MT0} after a slight modification of the construction.
      \item Teichm\"uller space: the $\mathrm{VMOA}$-Teichm\"uller space
      \[
      T_v=\mathrm{Aff}(\mathbb{R})\backslash \mathrm{SS}(\mathbb{R}).
      \]
   \end{itemize}
   \item Weil--Petersson class $\mathrm{WP}(\mathbb{R})$. 
   A quasisymmetric homeomorphism $h\colon \mathbb{R}\to \mathbb{R}$ belongs to the \emph{Weil--Petersson class} if it is locally absolutely continuous and $\log h'\in H^{1/2}(\mathbb{R})$, where $H^{1/2}(\mathbb{R})$ is the fractional Sobolev space.
   The collection of all such homeomorphisms is denoted by $\mathrm{WP}(\mathbb{R})$.
   \begin{itemize}
      \item Curve geometry: the Weil--Petersson quasicircles. 
      Bishop \cite{Bi25} described their geometry in terms of $\beta$-numbers, while the correspondence with conformal weldings was shown by Shen and Tang \cite{ST20}.
      \item Teichm\"uller space: the Weil--Petersson Teichm\"uller space 
      \[\mathcal{T}_0=\mathrm{Aff}(\mathbb{R})\backslash \mathrm{WP}(\mathbb{R}),\]
      this space was first investigated by Takhtajan and Teo in \cite{TT06}. 
      It carries a natural Hilbert manifold structure together with the Weil--Petersson metric, and plays an important role in Teichm\"uller theory, complex geometry, and mathematical physics.
   \end{itemize}
\end{itemize}

More generally, one may consider the $p$-Weil--Petersson classes $\mathrm{WP}_p(\mathbb{R})$ associated with the Besov regularity conditions. 

These subclasses of quasisymmetric homeomorphisms are closely related to the geometric behavior of the boundary curves and the regularity of the associated conformal weldings. 
They also admit analytic characterizations in terms of function spaces such as Bloch, little Bloch, $\mathrm{BMOA}$, $\mathrm{VMOA}$, Hardy spaces, and Besov spaces, and reflect increasingly refined geometric properties of the associated quasicircles and conformal mappings.

Motivated by these connections, it is natural to study conformal weldings associated with special classes of Jordan curves. 
Among all Jordan curves passing through $\infty$, graphs of real-valued continuous functions form one of the most natural and concrete classes. 

For the study of graph curves, we introduce the following notion of graph welding as a more suitable alternative to conformal welding. Let
\[
F(x)=x+if(x),
\qquad x\in\mathbb{R},
\]
where $f\colon \mathbb{R}\to\mathbb{R}$ is a continuous function, and let $\Gamma=F(\mathbb{R})$ be the graph of $f$.
Denote by $\Omega$ the domain lying above (or below) $\Gamma$, and let $G \colon \mathbb H\to\Omega$ be the conformal mapping with $G(\infty)=\infty$, meaning that $G(z) \to \infty$ as $z \to \infty$.
The boundary extension of $G$ induces an increasing homeomorphism
\[
\varphi = G^{-1}\circ F\colon \mathbb{R}\to\mathbb{R}.
\]
We call $\varphi$ the \emph{graph welding} associated with the function $f$. 
Once the choice of the domain above or below $\Gamma$ is fixed, the graph welding $\varphi$ is uniquely determined up to the pre-composition with an affine transformation.

Indeed, denote by $\Omega_+$ and $\Omega_-$ the domains lying above and below
$\Gamma$, respectively. Let
\[
G_+\colon\mathbb H\to\Omega_+,
\qquad
G_-\colon\mathbb H^*\to\Omega_-
\]
be conformal mappings satisfying $G_{\pm}(\infty)=\infty$. Define
\[
\varphi_+=G_+^{-1}\circ F,
\qquad
\varphi_-=G_-^{-1}\circ F,
\]
which are the graph weldings associated with $f$. 
It then follows that the conformal welding satisfies
\[
h=\varphi_-\circ\varphi_+^{-1},
\qquad
h^{-1}=\varphi_+\circ\varphi_-^{-1}.
\]

Consequently, the notion of graph welding provides a natural framework linking the analytic regularity of the graph function $f$, the geometry of the graph $\Gamma$, and the boundary behavior of conformal mappings.
The purpose of this paper is to investigate the graph weldings from the perspectives of absolute continuity, asymptotic symmetry, asymptotic smoothness, and the Weil--Petersson geometry.

The present work is closely related to the theory of quasisymmetric graphs, which was introduced and studied by Kovalev and Onninen \cite{KO}.
 
They showed, in particular, that the graph of every function in the classical Zygmund class is quasisymmetric, and more generally, that functions in a generalized Zygmund class associated with a doubling measure give rise to such graphs.  
 
A related viewpoint had already appeared in the work of Jerison and Kenig \cite{JK}: Zygmund domains are non-tangentially accessible domains, and in the plane this condition is equivalent, for simply connected domains, to the boundary being a quasicircle.

Our purpose is different from, but complementary to, this geometric theory.
The graph welding $\varphi=G^{-1}\circ F$ is precisely the change of parameter from the graph parametrization to the conformal parametrization. 
Thus, while the results of Kovalev and Onninen describe when the graph itself is quasisymmetric, the present paper studies how the analytic regularity of the graph function $f$ is reflected in the boundary homeomorphism $\varphi$. 
This leads naturally to finer classes of boundary correspondences such as symmetric, strongly quasisymmetric, strongly symmetric, and Weil--Petersson homeomorphisms.

The results in the $\mathrm{BMO}$ and $\mathrm{VMO}$ settings are also connected with the theory of chord-arc curves and $\mathrm{BMOA}$-Teichm\"uller spaces. 
A classical observation of Coifman and Meyer, used in the work of Jerison and Kenig \cite{JK, JK2}, asserts that graphs whose defining functions have $\mathrm{BMO}$ gradient satisfy a chord-arc type condition. 
This is one of the origins of the notion of $\mathrm{BMO}_1$ domains; their boundaries are the graphs of $f$ with $f'\in\mathrm{BMO}(\mathbb{R})$ in our case.
On the other hand, in the one-dimensional Teichm\"uller-theoretic setting mentioned above, there are characterizations of conformal welding via strongly (quasi)symmetric homeomorphisms.
Our theorems may be regarded as graph versions of these correspondences: under the $\mathrm{BMO}$ hypothesis on the derivative $f'$, the associated graph welding is strongly quasisymmetric; under the $\mathrm{VMO}$ hypothesis, it becomes strongly symmetric.
For the Weil--Petersson Teichm\"uller space, there is a similar correspondence between the classical setting and the graph version in terms of the Sobolev space $H^{1/2}$.

The paper is organized as follows.

Section~\ref{sec:preliminaries} recalls preliminary backgrounds on Bloch spaces, $\mathrm{BMOA}$ and $\mathrm{VMOA}$ spaces, Zygmund classes, asymptotically conformal and asymptotically smooth curves.

Section~\ref{sec:absolute_continuity} establishes localized versions of the Banach--Zaretskii theorem and the Riesz--Privalov theorem, and studies absolute continuity of graph weldings.

Section~\ref{sec:a_sufficient_condition_for_uniform_continuity} investigates a sufficient condition for the boundary extension of the conformal map to be uniformly continuous.

Section~\ref{sec:regularity_of_graph_weldings} is the core of the paper and contains our main results. We investigate graph weldings associated with Zygmund, $\mathrm{BMO}$, $\mathrm{VMO}$, Sobolev, and Besov regularity conditions and establish criteria for quasisymmetry, symmetry, strong quasisymmetry, strong symmetry, and the Weil--Petersson property. More precisely, we prove that Zygmund regularity of the graph function yields quasisymmetric graph weldings, that $\mathrm{BMO}$ and $\mathrm{VMO}$ regularity lead to strong quasisymmetry and strong symmetry, respectively, and that $H^{1/2}$ condition implies the Weil--Petersson property.

Finally, Section~\ref{sec:counterparts_in_conformal_welding} translates the results obtained for graph weldings into their counterparts for conformal weldings associated with graph curves. In particular, we derive criteria for conformal weldings to belong to the classes $\mathrm{QS}(\mathbb{R})$, $\mathrm{S}(\mathbb{R})$, $\mathrm{SQS}(\mathbb{R})$, $\mathrm{SS}(\mathbb{R})$, and $\mathrm{WP}(\mathbb{R})$.

\section{Preliminaries}\label{sec:preliminaries}

A locally integrable function $u \in L^1_{\mathrm{loc}}(\mathbb{R})$ is said to be of \emph{bounded mean oscillation} if
\[
\|u\|_{\mathrm{BMO}(\mathbb{R})}\coloneqq\sup_{I \subset \mathbb{R}}\frac{1}{|I|}\int_I |u(x) - u_I|\,dx< \infty,
\]
where the supremum is taken over all bounded intervals $I \subset \mathbb{R}$, 
and $u_I$ denotes the integral mean of $u$ over $I$. 
The collection of such functions is denoted by $\mathrm{BMO}(\mathbb{R})$, which becomes a Banach space after identifying functions that differ by a constant. 
A function $u \in \mathrm{BMO}(\mathbb{R})$ is said to be of \emph{vanishing mean oscillation} if
\[
\lim_{|I| \to 0}\frac{1}{|I|}\int_I |u(x) - u_I|\,dx= 0
\]
uniformly. 
The set of all such functions is denoted by $\mathrm{VMO}(\mathbb{R})$. 
This is a closed subspace of $\mathrm{BMO}(\mathbb{R})$.

Let $\mathbb{H}=\{z\in\mathbb{C}:\operatorname{Im} z>0\}$ be the upper half-plane. 
A holomorphic function $\Phi$ on the upper half-plane $\mathbb H$ is called a \emph{Bloch function} if it satisfies 
\[
\Vert \Phi \Vert_B\coloneqq\sup_{z \in \mathbb H}\,(\operatorname{Im} z)|\Phi'(z)| <\infty.
\]
The Banach space of all Bloch functions (modulo additive constants) with norm $\Vert \Phi \Vert_B$
is denoted by $B(\mathbb H)$. If $\Phi \in B(\mathbb H)$ vanishes at the boundary in the sense that 
\[
\lim_{\operatorname{Im} z \to 0_+}(\operatorname{Im} z)|\Phi'(z)| =0
\]
uniformly for $\operatorname{Re}\,z\in\mathbb{R}$, then $\Phi$ is called a \emph{little Bloch function}. The collection of all such functions forms a closed subspace of $B(\mathbb H)$, denoted by $B_0(\mathbb H)$. 
A holomorphic function $\Phi$ on $\mathbb{H}$ belongs to $\mathrm{BMOA}(\mathbb{H})$ if its boundary values, taken in the sense of non-tangential limits, lie in $\mathrm{BMO}(\mathbb{R})$. 
It belongs to $\mathrm{VMOA}(\mathbb{H})$ if, in addition, its boundary values lie in $\mathrm{VMO}(\mathbb{R})$. 
Equivalently, $\Phi \in \mathrm{BMOA}(\mathbb{H})$ (or $\mathrm{VMOA}(\mathbb{H})$) if and only if the measure $|\Phi'(z)|^{2} y\, dx\, dy$ is a Carleson measure (respectively, a vanishing Carleson measure) on $\mathbb{H}$.

A continuous function $f\colon \mathbb{R}\to\mathbb{R}$ is said to belong to the \emph{Zygmund class} $\Lambda_*(\mathbb{R})$ if
\[
\|f\|_{\Lambda_*(\mathbb{R})}\coloneqq\sup_{x\in\mathbb{R},\,h\neq 0}\frac{|f(x+h)+f(x-h)-2f(x)|}{|h|}<\infty .
\]
The quantity $\|f\|_{\Lambda_*(\mathbb{R})}$ is called the \emph{Zygmund seminorm} (or \emph{Zygmund constant}) of $f$. 
A function $f \in \Lambda_*(\mathbb{R})$ is said to belong to the \emph{little Zygmund class} $\lambda_*(\mathbb{R})$ if
\[
\lim_{h \to 0}\sup_{x \in \mathbb{R}}\frac{|f(x+h) + f(x-h) - 2f(x)|}{|h|}= 0.
\]
Strengthening this condition, we introduce a new class in the present paper:
we say that $f \in \Lambda_*(\mathbb{R})$ lies in the \emph{strongly little Zygmund class} $\lambda^{S}_*(\mathbb{R})$ if, for every $\varepsilon>0$ there exists $\delta=\delta(\varepsilon)>0$ such that for all $h\in\mathbb{R}$ with $|h|\leq \delta$, for all $x\in\mathbb{R}$, and for all $\lambda,\mu>0$ satisfying $\lambda+\mu=2$, one has
\[
\left|\lambda\bigl(f(x+\mu h)-f(x)\bigr)-\mu\bigl(f(x)-f(x-\lambda h)\bigr)\right|\leq \varepsilon |h|.
\]

We recall the definitions of the Muckenhoupt $A_{\infty}$ class and the fractional Sobolev space $H^{1/2}(\mathbb{R})$ for later use.
A locally integrable function $\omega\colon\mathbb{R}\to[0,\infty)$ is said to belong to $A_{\infty}$ if there exist constants $C>0$ and $\delta>0$ such that for every interval $I\subset\mathbb{R}$ and every measurable set $E\subset I$,
\[
\frac{\omega(E)}{\omega(I)}\leq C\left(\frac{|E|}{|I|}\right)^{\delta},
\]
where $\omega(E)=\int_E \omega(x)\,dx$ (see \cite{CF}).

The fractional Sobolev space $H^{1/2}(\mathbb{R})$ consists of all functions $u\in L^1_{\mathrm{loc}}(\mathbb{R})$ such that
\[
\iint_{\mathbb{R}\times\mathbb{R}}
\frac{|u(x)-u(y)|^2}{|x-y|^2}\,dx\,dy<\infty.
\]

A \emph{quasisymmetric embedding} is an embedding $F\colon X \to Y$ between metric spaces such that there exists an increasing homeomorphism $\eta\colon [0,\infty) \to [0,\infty)$ so that
\begin{equation}\label{eq:qs}
\frac{|x-a|}{|x-b|} \leq t \implies \frac{|F(x)-F(a)|}{|F(x)-F(b)|} \leq \eta(t)
\end{equation}
for all distinct points $x,a,b \in X$. Here, the distance in each metric space is simply denoted by $|a-b|$.
It follows directly from \eqref{eq:qs} that the inverse of a quasisymmetric embedding defined on its image is quasisymmetric and that the composition of quasisymmetric embeddings is again quasisymmetric.

Quasisymmetry was introduced by Beurling and Ahlfors \cite{BA} in their study of the boundary correspondence of quasiconformal self-maps of the upper half-plane. 
Later, Tukia and V\"ais\"al\"a \cite{TV} formulated a general definition of quasisymmetry in the setting of metric spaces.

In particular, if $F\colon \mathbb{R}\to\mathbb{C}$ is a quasisymmetric embedding, then its image $F(\mathbb{R})$ is a quasicircle (see \cite{Ah}). 
Equivalently, $F(\mathbb{R})$ can be realized as the image of the real line $\mathbb{R}$ under a quasiconformal self-map of the complex plane $\mathbb C$.

A quasisymmetric embedding $F\colon  X \to Y$ between metric spaces is called \emph{asymptotically symmetric} if for all $\varepsilon > 0$ and $t > 0$ there exists a $\delta > 0$ such that for all distinct points $x, a, b \in X$ contained in a ball of radius $\delta$, one has
\begin{equation*}
\frac{|x - a|}{|x - b|} \leq t \implies \frac{|F(x) - F(a)|}{|F(x) - F(b)|} \leq (1 + \varepsilon) t . 
\end{equation*}

The notion of asymptotically symmetric embeddings was first introduced by Brania and Yang \cite{BY}, who did not assume a priori that such embeddings are quasisymmetric.
In contrast, throughout this paper, we require an asymptotically symmetric embedding to be quasisymmetric.

Let $\Gamma$ be a Jordan curve passing through $\infty$. 
For distinct points $a, b \in \Gamma$, the subarc between $a$ and $b$ with smaller diameter is denoted by $\wideparen{ab}$, and its length is denoted by $\ell(\wideparen{ab})$ when it is rectifiable. 
A quasicircle $\Gamma$ is said to be \emph{asymptotically conformal} or \emph{asymptotically symmetric} if
\[
\lim_{|a-b| \to 0} \max_{w \in \wideparen{ab}} \frac{|a-w|+|w-b|}{|a-b|}=1
\]
for any $a,b \in \Gamma$. 
A locally rectifiable curve $\Gamma$ is called \emph{chord-arc} if there exists $C \geq 1$ such that $\ell(\wideparen{ab})/|a-b| \leq C$. 
This definition remains equivalent if the Euclidean distance $|a - b|$ is replaced by the spherical distance $d_{\widehat{\mathbb{C}}}(a,b)$. 
In particular, the chord-arc property for a Jordan curve $\Gamma \subset \widehat{\mathbb{C}}$ is invariant under M\"obius transformations (see \cite[p.~877]{Mac} and \cite[Section~7]{Tu}). 
A chord-arc curve $\Gamma$ is said to be \emph{asymptotically smooth} if the chord-arc constant tends to $1$ uniformly at small scales. 
More precisely, the curve $\Gamma$ satisfies
\[
\lim_{|a-b| \to 0} \frac{\ell(\wideparen{a b})}{|a - b|} = 1
\]
uniformly. 
The above definition of asymptotic smoothness is the natural extension to curves passing through $\infty$ of the corresponding notion for bounded Jordan curves introduced by Pommerenke \cite{Po78}.

\section{Absolute Continuity}\label{sec:absolute_continuity}

We recall several localized concepts on $\mathbb{R}$. Let $f\colon\mathbb{R}\to\mathbb{R}$ be a function.

\begin{enumerate}[(i)]
\item
We say that $f$ is \emph{locally absolutely continuous} if for every compact interval $[c,d]\subset\mathbb{R}$, the restriction $f|_{[c,d]}$ is absolutely continuous on $[c,d]$.

\item
We say that $f$ is \emph{locally of bounded variation} if for every compact interval $[c,d]\subset\mathbb{R}$, the total variation of $f$ on $[c,d]$, denoted by $V_c^d(f)$, is finite.

\item
We say that $f$ has Lusin's property $(N)$ if for every Lebesgue null set $A\subset\mathbb{R}$ (i.e.\ $|A|=0$), its image $f(A)$ also has Lebesgue measure zero, i.e.\ $|f(A)|=0$.
\end{enumerate}

The following theorem is a localized version of the classical Banach--Zaretskii theorem on $\mathbb{R}$. It follows immediately by applying the classical result to each compact interval $[c,d]\subset\mathbb{R}$.

\begin{proposition}[Localized Banach--Zaretskii theorem on $\mathbb{R}$]\label{local_BZ on R:label}
Let $f\colon\mathbb{R}\to\mathbb{R}$ be a continuous function. 
Then the following are equivalent:
\begin{enumerate}[(1)]
\item $f$ is locally absolutely continuous;

\item $f$ is locally of bounded variation and has Lusin's property $(N)$;

\item $f$ is differentiable almost everywhere, $f'\in L^1_{\mathrm{loc}}(\mathbb{R})$, and $f$ has Lusin's property $(N)$. 
\end{enumerate}

Moreover, under these equivalent conditions, for all $a,x\in\mathbb{R}$,
\[
f(x)=f(a)+\int_a^x f'(t)\,dt.
\]
\end{proposition}

We next generalize the classical Banach--Zaretskii theorem to rectifiable curves. 
Let $\Gamma\subset\mathbb{C}$ be a rectifiable curve, and let $\gamma\colon I\to\Gamma$ be an arc-length parametrization, where $I=[0,\ell(\Gamma)]\subset\mathbb{R}$.

\begin{proposition}[Banach--Zaretskii theorem on rectifiable curves]\label{BZ on rect curves:label}
Let $g\colon\Gamma\to\mathbb{R}$ be a continuous function. Then the following are equivalent:
\begin{enumerate}[(1)]
\item $g$ is absolutely continuous with respect to arc-length on $\Gamma$, 
i.e.\ $g\circ\gamma$ is absolutely continuous on $I$;

\item $g$ is of bounded variation with respect to arc-length on $\Gamma$, i.e.\ $g\circ\gamma$ is of bounded variation on $I$, and $g$ has Lusin's property $(N)$ with respect to arc-length measure;

\item $g\circ\gamma$ is differentiable almost everywhere on $I$, 
$(g\circ\gamma)'\in L^1(I)$, and $g$ has Lusin's property $(N)$ with respect to arc-length measure. 
\end{enumerate}

Moreover, under these equivalent conditions, for all $s,t\in I$,
\[
g(\gamma(t)) = g(\gamma(s)) + \int_s^t (g\circ\gamma)'(u)\,du.
\]
\end{proposition}

\begin{proof}
Since $\Gamma$ is rectifiable, it admits an arc-length parametrization $\gamma\colon I\to\Gamma$. 

Define $h\coloneqq g\circ\gamma\colon I\to\mathbb{R}$. Then $h$ is continuous on $I$. 
Moreover, absolute continuity, bounded variation, and Lusin's property $(N)$ on $\Gamma$ (with respect to arc-length measure) correspond exactly to the usual notions for $h$ on $I$.

Therefore, the result follows directly from the classical Banach--Zaretskii theorem applied to $h$ on $I$.
\end{proof}

Combining Propositions~\ref{local_BZ on R:label} and \ref{BZ on rect curves:label}, we obtain the Banach--Zaretskii theorem on locally rectifiable curves.

We next show that, under a regularity assumption, the inverse mapping inherits absolute continuity from the forward mapping.
\begin{lemma}\label{local inverse AC:label}
Let $G\colon\mathbb{R}\to\Gamma\subset\mathbb{C}$ be a homeomorphism. 
Assume that $G$ is locally absolutely continuous on $\mathbb{R}$ and $G^{-1}\colon\Gamma\to\mathbb{R}$ has Lusin's property $(N)$ with respect to arc-length measure on $\Gamma$. Then $G^{-1}$ is locally absolutely continuous on $\Gamma$.
\end{lemma}

\begin{proof}
Let $I\subset\mathbb{R}$ be a compact interval and let $\Gamma_I\coloneqq G(I)$. 
In particular, $G'$ exists almost everywhere on $I$ and $\Gamma_I$ is a rectifiable curve with length
\[
\ell(\Gamma_I)=\int_I |G'(t)|\,dt.
\]

Let $\gamma\colon[0,\ell(\Gamma_I)]\to\Gamma_I$ be an arc-length parametrization. 
Then there exists an increasing absolutely continuous function $\phi\colon I\to[0,\ell(\Gamma_I)]$ such that $G=\gamma\circ\phi$ on $I$, and hence $G^{-1}\circ\gamma=\phi^{-1}$.
Since $\phi$ is increasing and continuous, its inverse $\phi^{-1}$ is also continuous and increasing, and thus of bounded variation.

To prove the absolute continuity of $\phi^{-1}$, it suffices to verify that it has Lusin's property $(N)$. 
Let $A\subset[0,\ell(\Gamma_I)]$ be a Lebesgue null set. 
Since $\gamma$ is Lipschitz continuous, it maps Lebesgue null sets to arc-length null sets, and hence $\gamma(A)$ has arc-length measure zero in $\Gamma_I$. 
By Lusin's property $(N)$ of $G^{-1}$,
\[
|\phi^{-1}(A)| = |G^{-1}(\gamma(A))| = 0.
\]
Thus $\phi^{-1}$ satisfies Lusin's property $(N)$.

Applying the Banach--Zaretskii theorem, we conclude that $\phi^{-1}$ is absolutely continuous on $[0,\ell(\Gamma_I)]$. 
Therefore, $G^{-1}$ is absolutely continuous on $\Gamma_I$ by Proposition~\ref{BZ on rect curves:label}. 
\end{proof}

In the setting of conformal boundary parametrizations on the unit circle $\mathbb{T}$, the Banach--Zaretskii theorem reduces absolute continuity to bounded variation. 
Later, we will extend this claim to the real line $\mathbb{R}$.

\begin{lemma}[{\cite[Exercise 2, p.~138]{Pom}}]\label{ac-bv}
Let $\Phi\colon \mathbb{D}\to\Omega$ be a conformal mapping onto a bounded Jordan domain. 
Then
\[
\Phi|_{\mathbb{T}} \text{ is absolutely continuous }
\iff 
\Phi|_{\mathbb{T}} \text{ has bounded variation}.
\]
\end{lemma}

\begin{proof}
We include a brief proof for completeness. 
Since $\Omega$ is a Jordan domain, the Carath\'eodory theorem ensures that $\Phi$ extends to a homeomorphism from $\overline{\mathbb{D}}$ onto $\overline{\Omega}$.

If $\Phi|_{\mathbb{T}}$ is absolutely continuous, then it is of bounded variation.

Conversely, assume that $\Phi|_{\mathbb{T}}$ has bounded variation. 
Then the image curve $\Gamma=\Phi(\mathbb{T})$ is rectifiable, since its length is exactly the total variation of $t \mapsto \Phi(e^{it})$ on $[0,2\pi]$.
It follows from the Riesz--Privalov Theorem \cite[p.~134]{Pom} that the rectifiability of $\Gamma$ implies $\Phi'\in H^1(\mathbb{D})$, that is,
\[
\sup_{0<r<1}\int_0^{2\pi} |\Phi'(re^{it})|\,dt < \infty.
\]
In particular, the radial limits (in fact, the non-tangential limits) of $\Phi'$ exist almost everywhere and define an $L^1$ function $\phi$ on $[0,2\pi]$, given by
\[
\phi(t)=\lim_{r\to1^-} i e^{it} \Phi'(re^{it}).
\]
Consequently, for every $t\in[0,2\pi]$,
\[
\Phi(e^{it}) = \Phi(1) + \int_0^t \phi(s)\,ds,
\]
which shows that $\Phi|_{\mathbb{T}}$ is absolutely continuous.
\end{proof}

The Riesz--Privalov theorem \cite[Theorem 6.8]{Pom} characterizes rectifiability of the boundary in terms of the Hardy space. 
By passing from the unit disk to the upper half-plane, one obtains the following version of the Riesz--Privalov theorem for possibly unbounded Jordan domains.

Let $G$ be holomorphic in the upper half-plane $\mathbb H$. 
We say that $G$ belongs to the local Hardy space $H^1_{\mathrm{loc}}(\mathbb H)$ if for every compact interval $I\subset\mathbb{R}$,
\[
\limsup_{0<y<1}\int_I |G(x+iy)|\,dx<\infty.
\]

\begin{theorem}\label{RP-UHP:label}
Let $G\colon \mathbb H\to\Omega$ be a conformal mapping onto an unbounded Jordan domain $\Omega$ with $G(\infty)=\infty$, and let $\Gamma=\partial\Omega$. 
Then the following are equivalent:
\begin{enumerate}
\item $\Gamma$ is locally rectifiable;
\item $G'\in H^1_{\mathrm{loc}}(\mathbb H)$.
\end{enumerate}

Moreover, if these conditions hold, then for every measurable set $E\subset\mathbb{R}$,
\[
|E|=0 \iff \ell(G(E))=0,
\]
where $\ell$ also denotes the arc-length measure.
\end{theorem}

\begin{proof}
Assume first that $G'\in H^1_{\mathrm{loc}}(\mathbb H)$. Let $I=[a,b]\subset\mathbb{R}$ be a compact interval, and let
\[
L=\{a=x_0<x_1<\cdots<x_n=b\}
\]
be a partition of $I$.

By the Carath\'eodory theorem, $G$ extends continuously to $\overline{\mathbb H}$. Hence
\begin{align*}
   \sum_{k=1}^n |G(x_k)-G(x_{k-1})|
&=\limsup_{y\to0}\sum_{k=1}^n|G(x_k+iy)-G(x_{k-1}+iy)|\\
&\leq\limsup_{y\to0}\int_I |G'(x+iy)|\,dx.
\end{align*}
Combining this with the assumption $G'\in H^1_{\mathrm{loc}}(\mathbb H)$ yields
\[
\sum_{k=1}^n |G(x_k)-G(x_{k-1})|
\leq
\limsup_{0<y<1}\int_I |G'(x+iy)|\,dx
<\infty.
\]
Taking the supremum over all partitions of $I$, we obtain $\ell(G(I))<\infty$. 
Therefore, $\Gamma$ is locally rectifiable.

Conversely, assume that $\Gamma$ is locally rectifiable. 
Let $I=[a,b]\subset\mathbb{R}$ be a compact interval, and define
\[
\Gamma_1\coloneqq G((-\infty,a)),
\qquad
\Gamma_2\coloneqq G((b,+\infty)).
\]
By choosing points $w_1, w_2 \in \Omega$ sufficiently far from $G([a,b])$ and sufficiently close to $\Gamma_1$ and $\Gamma_2$ respectively, we can find $\xi_i\in\Gamma_i$ such that
\[
|w_i-\xi_i|=\operatorname{dist}(w_i,\Gamma)=\operatorname{dist}(w_i,\Gamma_i).
\]
Then the segments $[w_i,\xi_i)$ lie entirely in $\Omega$.

Let $z_i=G^{-1}(w_i)\in\mathbb H$ for $i=1,2$, and let $S$ denote the hyperbolic geodesic in $\mathbb H$ joining $z_1$ and $z_2$. 
Set $J=[G^{-1}(\xi_1),G^{-1}(\xi_2)]$, where, 
\[G^{-1}(\xi_1)<a<b<G^{-1}(\xi_2).\] 
It is clear that $J$ is a compact interval containing $I$. 
Furthermore, $G(S)$ is a rectifiable arc in $\Omega$ connecting $w_1$ and $w_2$. 
It follows that the closed curve $\Gamma_I$ formed by
\[
G(J),\quad [w_1,\xi_1],\quad [w_2,\xi_2],\quad \text{and }\ G(S)
\]
is a rectifiable Jordan curve. 
Denote by $\Omega_I$ the Jordan domain bounded by $\Gamma_I$.

Let $T\colon \mathbb D\to G^{-1}(\Omega_I)$ be conformal. 
Then $G\circ T$ maps $\mathbb D$ conformally onto $\Omega_I$. 
By the classical Riesz--Privalov theorem for the unit disk, we have $(G\circ T)'\in H^1(\mathbb D)$.

Since $J$ is an analytic boundary arc of $G^{-1}(\Omega_I)$, the Schwarz reflection principle (see, for example, \cite[p.~14]{Dur}) implies that $T$ extends analytically across the arc $T^{-1}(J)$. 
In particular, $T$ is bi-Lipschitz in a neighborhood of $T^{-1}(I)$. 
More precisely, there exist constants $\kappa\ge1$ and $y_0>0$ such that
\[
\kappa^{-1}\leq |T'(\zeta)|\leq \kappa,
\qquad
\zeta\in T^{-1}(I\times[0,y_0]).
\]

For sufficiently small $y_0>0$ and each $0\leq y\leq y_0$, define
\[
I_y=\{x+iy:x\in I\}.
\]
Then each curve $T^{-1}(I_y)$ admits a parametrization
\[
\zeta(\theta,y)=r_y(\theta)e^{i\theta},
\qquad
\theta\in[\alpha_y,\beta_y], \qquad \frac12\leq r_y(\theta)\le1.
\]
Write $z(\theta,y)=T(\zeta(\theta,y))=x(\theta,y)+iy$. 
Differentiating with respect to $\theta$ gives
\[
\left|\frac{d\zeta}{d\theta}(\theta,y)\right|=\frac{\left|\tfrac{dx}{d\theta} (\theta,y)\right|}{|T'(\zeta)|}\le\kappa \left|\frac{dx}{d\theta} (\theta,y)\right|.
\]
Observe that
\[
(\theta,y)\mapsto x(\theta,y)
\coloneqq
\operatorname{Re}T(\zeta(\theta,y))
\]
is a $C^1$ map on the compact set $T^{-1}(I\times[0,y_0])$. 
Hence $|\tfrac{dx}{d\theta}(\theta,y)|$ is uniformly bounded there. 
Consequently, there exists a constant $M>0$ (depending on $I$) such that $\left|\frac{d\zeta}{d\theta}\right|\leq M$.

Using the change of variables formula and the preceding estimate, we obtain
\begin{align*}
\int_{I_y}|G'(z)|\,|dz|&=\int_{T^{-1}(I_y)}|G'(T(\zeta))|\,|T'(\zeta)|\,|d\zeta|\leq M\int_{\alpha_y}^{\beta_y}|(G\circ T)'(r_y(\theta)e^{i\theta})|\,d\theta\\
&\leq M\int_0^{2\pi}(G\circ T)'_*(e^{i\theta})\,d\theta\leq M\|(G\circ T)'\|_{H^1(\mathbb D)},
\end{align*}
where $(G\circ T)'_*$ denotes the non-tangential maximal function of $(G\circ T)'$. 
The last inequality follows from \cite[Theorem 3.1]{Ga}. 
This proves $G'\in H^1_{\mathrm{loc}}(\mathbb{H})$.

Now assume that $\Gamma$ is locally rectifiable, and let $I$ and $\Omega_I$ be as above. 
If $E\subset I$ satisfies $|E|=0$, then the bi-Lipschitz property of $T$ near $T^{-1}(I)$ implies that $|T^{-1}(E)|=0$. 
Applying the Riesz--Privalov theorem in $\mathbb D$ yields
\[
\ell(G(E))=\ell\bigl((G\circ T)(T^{-1}(E))\bigr)=0.
\]

Conversely, if $\ell(G(E))=0$, then
\[
\ell\bigl((G\circ T)(T^{-1}(E))\bigr)=\ell(G(E))=0,
\]
and again by the Riesz--Privalov theorem for $\mathbb D$, it follows that $|T^{-1}(E)|=0$ and hence $|E|=0$. 
If $E$ is unbounded, decompose $E$ into a countable union of bounded null sets and apply the above argument to each of them.
\end{proof}

\begin{remark}
The localization argument used in the proof of Theorem~\ref{RP-UHP:label} also yields the existence of the non-tangential limit of $G'$ and its local $L^1$ convergence to these boundary values.
More precisely,
\begin{align*}\int_I |G'(x+iy)-G'(x)|\,dx \to 0
\qquad \text{as } y\to0_+
\end{align*}
for every compact interval $I\subset\mathbb{R}$.
\end{remark}

As an application of Theorem~\ref{RP-UHP:label}, we extend the result of Lemma~\ref{ac-bv} to the unbounded case.
\begin{corollary}\label{bv=ac:label}
   Let $G\colon \mathbb{H}\to\Omega$ be a conformal mapping onto an unbounded Jordan domain $\Omega$ with $G(\infty)=\infty$. 
   Then
\[
G|_{\mathbb{R}} \text{ is locally absolutely continuous }
\iff 
G|_{\mathbb{R}} \text{ has locally bounded variation}.
\]
\end{corollary}
\begin{proof}
If $G|_{\mathbb{R}}$ is locally absolutely continuous, then it has locally bounded variation on $\mathbb{R}$.

Conversely, assume that $G|_{\mathbb{R}}$ has locally bounded variation. 
Then the image curve $\Gamma=G(\mathbb{R})$ is locally rectifiable. 
By Theorem~\ref{RP-UHP:label}, $G|_{\mathbb{R}}$ satisfies Lusin's property $(N)$. 
Hence, the localized Banach--Zaretskii theorem (Proposition~\ref{local_BZ on R:label}) implies that $G|_{\mathbb{R}}$ is locally absolutely continuous.
\end{proof}

Recall that for a continuous function $f\colon \mathbb{R} \to \mathbb{R}$, we let $F(x) = x + i f(x)$ be the parametrization of its graph $\Gamma = F(\mathbb{R})$. 
Let $\Omega$ denote the domain lying above (or below) $\Gamma$. There exists a conformal mapping $G\colon \mathbb{H} \to \Omega$ with $G(\infty) = \infty$. 
The boundary extension of $G$ induces a homeomorphism
\[\varphi \coloneqq G^{-1} \circ F\colon \mathbb{R} \to \mathbb{R},\]
which we call the graph welding associated with $f$.

We first establish that the absolute continuity of $f$ implies that of the associated graph welding $\varphi$.
\begin{theorem}\label{ac_f:label}
   Let $f$ be a locally absolutely continuous function. 
   Then the associated graph welding $\varphi$ is locally absolutely continuous. 
\end{theorem}
\begin{proof}
Since $f$ is locally absolutely continuous, the curve $\Gamma$ is locally rectifiable. 
This implies that the boundary parametrization $G|_{\mathbb{R}}$ of $\Gamma$ is of locally bounded variation. 
By Corollary~\ref{bv=ac:label}, it follows that $G|_{\mathbb{R}}$ is locally absolutely continuous. 
Furthermore, Theorem~\ref{RP-UHP:label} implies that the inverse mapping $G^{-1}|_{\Gamma}$ has Lusin's property $(N)$ with respect to arc-length. 
Combining this with Lemma~\ref{local inverse AC:label}, we conclude that $G^{-1}|_{\Gamma}$ is locally absolutely continuous with respect to arc-length.

Since $\varphi$ is an increasing homeomorphism, it has locally bounded variation. 
Moreover, Lusin's property $(N)$ is preserved under composition with absolutely continuous mappings, and hence $\varphi$ inherits Lusin's property $(N)$ from $G^{-1}$ and $F$. 
Therefore, by the Banach--Zaretskii theorem, $\varphi$ is locally absolutely continuous. 
\end{proof}

\begin{theorem}
   Let $f$ be a continuous function, and suppose that the associated graph welding $\varphi$ is locally absolutely continuous. 
   Then $F|_{\mathbb{R}}$ is locally absolutely continuous if and only if $G|_{\mathbb{R}}$ is locally absolutely continuous. 
\end{theorem}
\begin{proof} 
   Since $F=G\circ \varphi$ and $\varphi$ is an increasing locally absolutely continuous homeomorphism of $\mathbb{R}$, the proof follows from the definition of absolutely continuous functions.
\end{proof}

\section{A Sufficient Condition for Uniform Continuity}\label{sec:a_sufficient_condition_for_uniform_continuity}

We first provide a sufficient condition for the boundary extension $G|_{\mathbb{R}}\colon \mathbb{R} \to \Gamma$ of a conformal map $G\colon \mathbb H \to \Omega$ to be uniformly continuous.

\begin{lemma}\label{strip}
If the imaginary part of $\Gamma$ is bounded (or more generally, if $\Gamma$ is contained in a bounded strip domain of $\mathbb C$), then both $G|_{\mathbb{R}}$ and $G^{-1}|_{\Gamma}$ are uniformly continuous.
\end{lemma}

\begin{proof}
By composing with a translation and a rotation, we may assume that $G$ is a conformal mapping of $\mathbb H$ into $\mathbb H$ such that $G(\infty) =\infty$.
By the Julia--Carath\'eodory theorem (see \cite[Section 4.5]{Sha})
or the Nevanlinna representation for Pick functions (see \cite[Section 5.3]{RR}), the angular derivative of $G$ at $\infty$, defined by
\[
c=\lim_{y \to \infty}\frac{\operatorname{Im}\ G(x+iy)}{y},
\]
exists for any $x \in \mathbb{R}$, and satisfies $0<c<\infty$. 
For any $y_0>0$, define the half-plane
\[
{\mathbb H}(y_0)=\{z=x+iy \in \mathbb H \mid y>y_0\}.
\]
Then, we have $G({\mathbb H}(y_0)) \subset {\mathbb H}(cy_0)$.
Indeed, by applying the Nevanlinna representation, $G$ can be written as
\[
G(z)=\operatorname{Re}\,G(i)+cz+\int_{-\infty}^\infty \left(\frac{1}{\xi-z}-\frac{\xi}{1+\xi^2}\right)d\mu(\xi)
\]
for some Borel measure $\mu$ on $\mathbb{R}$ satisfying $\int_{\mathbb{R}}(1+\xi^2)^{-1}d\mu(\xi)<\infty$.
This directly implies that $\operatorname{Im}\,G(z) \geq c\, \operatorname{Im}\,z$.

To prove the uniform continuity of $G^{-1}|_{\Gamma}$, we choose $y_0>0$ large enough so that $cy_0 >\sup_{w \in \Gamma} \operatorname{Im}\, w$. 
Since $G({\mathbb H}(y_0)) \subset {\mathbb H}(cy_0)$, this implies that $d(G(x+iy_0),\Gamma)$ is uniformly bounded away from $0$. 
Hence $G^{-1}|_{\Gamma}$ is uniformly continuous \cite[Lemma 4.2]{MT1}.

For the uniform continuity of $G|_{\mathbb{R}}$, we restrict $G^{-1}$ to ${\mathbb H}(m)$, where $m=\sup_{w \in \Gamma} \operatorname{Im}\, w$, and regard $G^*(w)=G^{-1}(w+m)$ as a conformal mapping of $\mathbb H$ into itself. 
Then, the angular derivative of $G^*$ at $\infty$ is $1/c$, and we have
\[
G^{-1}({\mathbb H}(m+1)) =G^*({\mathbb H}(1)) \subset {\mathbb H}(1/c).
\]
This yields
\[
0<\operatorname{Im}\,G(x+i/c) \leq m+1,
\]
which shows that $d(G(x+i/c),\Gamma) \leq m+1$ for all $x \in \mathbb{R}$. 
By \cite[Lemma 3.2]{MT1}, $G|_{\mathbb{R}}$ is uniformly continuous.
\end{proof}

\begin{theorem}
Suppose that the imaginary part of $\Gamma$ is bounded. 
Then $\Gamma$ is an asymptotically conformal quasicircle if and only if
$\log G' \in B_0(\mathbb H)$.
\end{theorem}

\begin{proof}
This is a direct consequence of \cite[Theorems 3.1 and 4.1]{MT1} and Lemma~\ref{strip}.
\end{proof}

\section{Regularity of Graph Weldings}\label{sec:regularity_of_graph_weldings}
Throughout this section, let $\Gamma$ be the graph of a continuous function $f$, and define the parametrization $F\colon \mathbb{R}\to \Gamma$ by $F(x)=x+if(x)$. 
Let $G\colon \mathbb{H}\to \Omega$ denote a conformal mapping from the upper half-plane $\mathbb{H}$ onto the Jordan domain $\Omega$ lying above (or below) $\Gamma$ with $G(\infty)=\infty$.

In order to prove the main theorems of this section, we require several lemmas. The first lemma is crucial in establishing a sufficient condition for $F$ to be asymptotically symmetric.
\begin{lemma}\label{SLZ-F-as:label}
Assume that $f\colon\mathbb R\to\mathbb R$ belongs to the strongly little Zygmund class $\lambda_*^S(\mathbb R)$. Let $F$ and $\Gamma$ be as above. 
Then $F$ is an asymptotically symmetric homeomorphism of $\mathbb{R}$ onto $\Gamma$.
\end{lemma}
\begin{proof}
Fix $x\in\mathbb{R}$ and $h\neq 0$. Let $\lambda,\mu>0$ with $\lambda+\mu=2$. Set
\[
\Delta_+^\lambda = f(x+\lambda h)-f(x), 
\quad 
\Delta_-^\mu = f(x)-f(x-\mu h).
\]
The strongly little Zygmund condition implies that for every $\eta>0$ there exists $\delta(\eta)>0$ such that for $|h|\le\delta(\eta)$,
\[
|\, \mu\Delta_+^\lambda - \lambda\Delta_-^\mu \,| \leq \eta |h|.
\]
In particular,
\begin{equation}\label{eq:basic_as}
|\mu\Delta_+^\lambda| \leq |\lambda\Delta_-^\mu| + \eta|h|.
\end{equation}

Observe that
\begin{align*}
   \mu^2|F(x+\lambda h)-F(x)|^2&= \lambda^2\mu^2 h^2 + \mu^2(\Delta_+^\lambda)^2,\\
\lambda^2|F(x)-F(x-\mu h)|^2&= \lambda^2\mu^2 h^2 + \lambda^2(\Delta_-^\mu)^2.
\end{align*}
Hence
\[
\left|\frac{\mu\bigl(F(x+\lambda h)-F(x)\bigr)}{\lambda\bigl(F(x)-F(x-\mu h)\bigr)}\right|^2=\frac{\lambda^2\mu^2 h^2 + \mu^2(\Delta_+^\lambda)^2}{\lambda^2\mu^2 h^2 + \lambda^2(\Delta_-^\mu)^2}.
\]
Using \eqref{eq:basic_as}, we estimate
\[
\mu^2(\Delta_+^\lambda)^2 \leq (|\lambda\Delta_-^\mu|+\eta|h|)^2= \lambda^2(\Delta_-^\mu)^2 + 2\lambda\eta|\Delta_-^\mu||h| + \eta^2 h^2.
\]
Therefore,
\begin{align}\label{eq:ratio}
   \left|\frac{\mu\bigl(F(x+\lambda h)-F(x)\bigr)}{\lambda\bigl(F(x)-F(x-\mu h)\bigr)}\right|^2
   &\leq\frac{\lambda^2\mu^2 h^2 + \lambda^2(\Delta_-^\mu)^2 + 2\lambda\eta |\Delta_-^\mu||h| + \eta^2 h^2}{\lambda^2\mu^2 h^2 + \lambda^2(\Delta_-^\mu)^2} \notag \\
   &=1+\frac{2\lambda\eta |\Delta_-^\mu||h|+\eta^2 h^2}{\lambda^2\bigl(\mu^2h^2+|\Delta_-^\mu|^2\bigr)}.
\end{align}

Set $d=|\Delta_-^\mu|/(\mu|h|)\geq 0$.
Dividing the numerator and denominator in \eqref{eq:ratio} by $\lambda^2\mu^2h^2$, it follows that
\[
\left|\frac{\mu\bigl(F(x+\lambda h)-F(x)\bigr)}{\lambda\bigl(F(x)-F(x-\mu h)\bigr)}\right|^2\leq 1+\frac{2(\eta/(\lambda \mu))\, d + (\eta/(\lambda \mu))^2}{1+d^2} \eqqcolon  g_\eta(d).
\]
For fixed $\lambda,\mu>0$ with $\lambda+\mu=2$, the function $g_\eta$ is bounded on $[0,\infty)$ and satisfies
\[
\lim_{d\to\infty} g_\eta(d)=1,\quad g_\eta(0)=1+\frac{\eta^2}{\lambda^2\mu^2}.
\]
Hence
\[
\sup_{d\ge0} g_\eta(d)=\left(\frac{\eta+\sqrt{\eta^2+4\lambda^2\mu^2}}{2\lambda \mu}\right)^2.
\]

Then for every $\varepsilon>0$, set
\[
\eta =\frac{\lambda\mu\varepsilon(\varepsilon+2)}{1+\varepsilon}
\]
so that for $|h|\leq\delta(\eta)$, it follows that
\[
\left|\frac{\mu\bigl(F(x+\lambda h)-F(x)\bigr)}{\lambda\bigl(F(x)-F(x-\mu h)\bigr)}\right|\leq \frac{\eta+\sqrt{\eta^2+4\lambda^2\mu^2}}{2\lambda\mu}= 1+\varepsilon,
\]
uniformly for $x\in\mathbb{R}$ with $\lambda+\mu=2$. 

Therefore, for all $|h|\leq \delta=\delta(\eta)$, we have
\[
\left|\frac{F(x+\lambda h)-F(x)}{F(x)-F(x-\mu h)}\right|\leq (1+\varepsilon)\frac{\lambda}{\mu}
\]
uniformly in $x$.
This is precisely the asymptotic symmetry condition. 
Hence $F$ is an asymptotically symmetric homeomorphism of $\mathbb{R}$ onto $\Gamma$.
\end{proof}

The next lemma demonstrates that the inverse of an asymptotically symmetric homeomorphism inherits the same property, provided it is uniformly continuous. 

We denote the \emph{modulus of continuity} of $f$ by
\[
\omega(t)=\omega(t;f)=\sup\{|f(x)-f(y)|:|x-y|\leq t\}.
\]

\begin{lemma}\label{Finverse_AS:label}
Let $F\colon \mathbb{R} \to \Gamma\subset \mathbb{C}$ be an asymptotically symmetric homeomorphism.
If $F^{-1}$ is uniformly continuous, then $F^{-1}$ is asymptotically symmetric.
\end{lemma}

\begin{proof}
Given $\varepsilon'>0$ and $t'>0$, choose $\varepsilon=\varepsilon'/2$ and set
\[
t = \frac{1}{(1+\varepsilon')t'}.
\]
Since $F$ is asymptotically symmetric, there exists $\delta>0$ such that for all $a,b,x$ in a ball of radius $\delta$,
\[
\left|\frac{b-x}{a-x}\right|\leq t\implies
\left|\frac{b'-x'}{a'-x'}\right|\leq (1+\varepsilon)t,
\]
where \[a'=F(a),\qquad b'=F(b),\qquad  x'=F(x).\]

Let $\omega$ be a modulus of continuity of $F^{-1}$, and choose $\delta'>0$ such that $\omega(2\delta')< \delta$.  
If $a',b',x'\in\Gamma$ lie in a ball of radius $\delta'$, then their preimages lie in a ball of radius $\delta$.  
Taking the contrapositive of the above implication gives
\[
\left|\frac{a'-x'}{b'-x'}\right|\leq t'<\frac{1}{(1+\varepsilon)t}
\implies
\left|\frac{a-x}{b-x}\right|\leq (1+\varepsilon')t'.
\]
Hence, $F^{-1}$ is asymptotically symmetric.
\end{proof}

Our first theorem is a basic result, which translates a well-known property of a Zygmund function into the graph welding setting.

\begin{theorem}\label{qs:label}
Let $f$ be a Zygmund function on $\mathbb{R}$. 
Then the associated graph welding $\varphi$ belongs to $\mathrm{QS}(\mathbb{R})$.
\end{theorem}

\begin{proof}
By a result of Jerison and Kenig \cite[p.~92]{JK}, a planar Zygmund domain is an NTA domain, and in the plane the NTA condition is equivalent to the boundary being a quasicircle. 
In particular, the graph $\Gamma$ of a Zygmund function $f$ on $\mathbb{R}$ is a quasicircle passing through $\infty$. See also \cite[Theorem 1.4]{KO}.
By the classical quasiconformal theory, the boundary extension
$G|_{\mathbb R}\colon\mathbb R\to\Gamma$ is quasisymmetric (see, e.g. \cite{Ah}), and its inverse $G^{-1}\colon \Gamma\to \mathbb{R}$ is quasisymmetric as well.
Since the composition of quasisymmetric mappings remains quasisymmetric, it follows that the graph welding $\varphi = G^{-1}\circ F$ is quasisymmetric.
\end{proof}

The next theorem provides a sufficient condition for the graph welding associated with a Zygmund function to be symmetric.
Unlike the case of quasisymmetric homeomorphisms, the composition of asymptotically symmetric homeomorphisms does not possess the same property automatically. 
To deal with this issue, we impose the additional assumption $f\in L^{\infty}(\mathbb{R})$ on the claim, although this condition is probably far from sharp.

\begin{theorem}\label{s:label}
Let $f\colon \mathbb{R}\to \mathbb{R}$ be a Zygmund function. 
Assume that $f$ satisfies the following conditions:
\begin{enumerate}
   \item $f\in L^{\infty}(\mathbb{R})$;
   \item $f\in \lambda^{S}_{*}(\mathbb{R})$.
\end{enumerate}
Then the associated graph welding $\varphi$ belongs to $\mathrm{S}(\mathbb{R})$.
\end{theorem}

\begin{proof}
Since $f \in L^\infty(\mathbb{R})$, the graph $\Gamma$ is contained in a bounded horizontal strip. 
By Lemma~\ref{strip}, both $G|_{\mathbb{R}}$ and $G^{-1}|_{\Gamma}$ are uniformly continuous. 
Moreover, \cite[Theorem~3.1]{MT1} implies that $G|_{\mathbb{R}}$ is an asymptotically symmetric embedding onto $\Gamma$. 
Since $G^{-1}|_{\Gamma}$ is uniformly continuous, it follows from Lemma~\ref{Finverse_AS:label} that $G^{-1}|_{\Gamma}$ is also asymptotically symmetric.

On the other hand, the assumption $f\in \lambda_*^{S}(\mathbb{R})$ implies that $F$ is an asymptotically symmetric embedding by Lemma~\ref{SLZ-F-as:label}. 
Hence, for any $\varepsilon_1>0$ and $t_1>0$, there exists $\delta_1>0$ such that
\begin{equation}\label{F_as_new}
\frac{|x-a|}{|x-b|}\leq t_1 \quad \implies \quad \frac{|F(x)-F(a)|}{|F(x)-F(b)|} \leq (1+\varepsilon_1)t_1
\end{equation}
whenever $x,a,b\in\mathbb{R}$ lie in a ball of radius $\delta_1$.

Similarly, for any $\varepsilon_2>0$ and $t_2>0$, there exists $\delta_2>0$ such that
\begin{equation}\label{G_inverse_as_new}
\frac{|\tilde x-\tilde a|}{|\tilde x-\tilde b|}\leq t_2 \quad \implies\quad\frac{|G^{-1}(\tilde x)-G^{-1}(\tilde a)|}{|G^{-1}(\tilde x)-G^{-1}(\tilde b)|}\leq (1+\varepsilon_2)t_2
\end{equation}
whenever $\tilde x,\tilde a,\tilde b\in\Gamma$ lie in a ball of radius $\delta_2$.

Furthermore, $F$ is uniformly continuous, since every little Zygmund function satisfies the modulus of continuity estimate $o(h\log(1/h))$ as $h\to0$; see, for example, \cite{Zy}. Let $\omega(t)$ be the modulus of continuity of $F$. 
Choose $\delta'>0$ such that $\omega(2\delta')<\delta_2$.

Now let $0<\varepsilon<1$ and $t>0$ be given. Set
\[
t_1=t,\qquad\varepsilon_1=\frac{\varepsilon}{3},\qquad 
t_2=(1+\varepsilon_1)t_1,\qquad\varepsilon_2=\varepsilon_1,
\]
and define $\delta=\min\{\delta_1,\delta'\}$.

Suppose that $x,a,b\in\mathbb{R}$ lie in a ball of radius $\delta$ and satisfy
\[
\frac{|x-a|}{|x-b|}\leq t.
\]
Then \eqref{F_as_new} yields
\[
\frac{|F(x)-F(a)|}{|F(x)-F(b)|}\leq (1+\varepsilon_1)t_1=t_2.
\]

Since $x,a,b$ lie in a ball of radius $\delta$, we have
\[
|x-a|<2\delta\leq 2\delta',\qquad|x-b|<2\delta\leq 2\delta'.
\]
Hence
\[
|F(x)-F(a)|\leq \omega(2\delta')<\delta_2,
\]
and similarly
\[
|F(x)-F(b)|<\delta_2.
\]
Therefore $F(x),F(a),F(b)\in\Gamma$ lie in a ball of radius $\delta_2$, and so \eqref{G_inverse_as_new} implies that
\[
\frac{|G^{-1}\circ F(x)-G^{-1}\circ F(a)|}{|G^{-1}\circ F(x)-G^{-1}\circ F(b)|}\le
(1+\varepsilon_2)t_2=(1+\varepsilon_1)^2 t_1\leq (1+\varepsilon)t.
\]
Hence $\varphi=G^{-1}\circ F$ is asymptotically symmetric. 
In particular, it is symmetric.
\end{proof}

\begin{remark}
We also see that $\varphi^{-1}=F^{-1} \circ G|_{\mathbb{R}}$ is in $\mathrm{S}(\mathbb{R})$.
Since $\varphi \in \mathrm{S}(\mathbb{R})$ and $\varphi^{-1}$ is uniformly continuous, this follows from \cite[Corollary 3.4]{WM23}.
\end{remark}

We next investigate graph weldings involving $\mathrm{BMO}$, $\mathrm{VMO}$, and $H^{1/2}$.
To this end, we first establish several auxiliary lemmas.

The first lemma was originally proved by Stein and Zygmund
\cite[p.~343]{SZ}.
For the convenience of the reader, we include a simpler proof, which also applies to the $\mathrm{VMO}$--little Zygmund correspondence.

\begin{lemma}\label{bmo-zyg}
Let $f\colon  \mathbb{R} \to \mathbb{R}$ be locally absolutely continuous. 
If $f' \in \mathrm{BMO}(\mathbb{R})$ (respectively, $\mathrm{VMO}(\mathbb{R})$), then $f\in \Lambda_*(\mathbb{R})$ (respectively, $\lambda_*(\mathbb{R})$).
\end{lemma}

\begin{proof}
Since $f$ is locally absolutely continuous, the fundamental theorem of calculus yields
\[
f(x+h) + f(x-h) - 2f(x)= \int_{0}^{h} \bigl(f'(x+t) - f'(x-t)\bigr)\,dt.
\]

To estimate this integral, let $I = [x-h, x+h]$ and denote the mean of $f'$ over $I$ by
\[
(f')_I \coloneqq \frac{1}{2h}\int_{x-h}^{x+h} f'(t)\,dt.
\]
Then
\[
f'(x+t) - f'(x-t)= \bigl(f'(x+t)-(f')_I\bigr) - \bigl(f'(x-t)-(f')_I\bigr),
\]
and hence
\begin{align*}
|f(x+h) + f(x-h) - 2f(x)|&\leq \int_{0}^{h} |f'(x+t)-(f')_I|\,dt+ \int_{0}^{h} |f'(x-t)-(f')_I|\,dt\\
   &=\int_{x-h}^{x+h} |f'(t)-(f')_I|\,dt .
\end{align*}

By the definition of the $\mathrm{BMO}$ norm, this last term is bounded by $2h \|f'\|_{\mathrm{BMO}(\mathbb{R})}$. 
Therefore,
\[
|f(x+h) + f(x-h) - 2f(x)|\leq 2h\,\|f'\|_{\mathrm{BMO}(\mathbb{R})}.
\]
This shows that $f$ belongs to the Zygmund class $\Lambda_{*}(\mathbb{R})$.

If, in addition, $f'\in \mathrm{VMO}(\mathbb R)$, then the mean oscillation of $f'$ over $[x-h,x+h]$ tends to zero as $h\to0$, uniformly in $x$. 
The same argument yields
\[
\frac{|f(x+h)+f(x-h)-2f(x)|}{h} \to 0,
\]
showing that $f$ belongs to the little Zygmund class $\lambda_{*}(\mathbb{R})$.
\end{proof}

The next lemma describes the corresponding regularity of the graph $\Gamma$
under the assumptions $f'\in\mathrm{BMO}(\mathbb R)$ and $f'\in\mathrm{VMO}(\mathbb R)$.
The first assertion is a special case of a more general result of
Jerison and Kenig \cite[p.~92]{JK}. 
Nevertheless, for the convenience of the reader, we sketch a direct proof in the planar case, from which the second assertion also follows. 

\begin{lemma}\label{bmo1=ca:label}
Let $f\colon \mathbb{R}\to\mathbb{R}$ be locally absolutely continuous. 
If $f' \in \mathrm{BMO}(\mathbb{R})$, then the graph $\Gamma$ is a chord-arc curve. 
Moreover, if $f' \in \mathrm{VMO}(\mathbb{R})$, then $\Gamma$ is an asymptotically smooth chord-arc curve.
\end{lemma}

\begin{proof}
For any interval $I=[a,b]$, denote $(f')_I=\tfrac{1}{|I|}\int_I f'(x)\,dx$.
By the fundamental theorem of calculus, the chord length of $\Gamma_I$ is
\begin{align*}
\sqrt{|I|^2+(f(b)-f(a))^2}= \sqrt{|I|^2+|I|^2\big((f')_I\big)^2} = |I|\sqrt{1+\big((f')_I\big)^2}.
\end{align*}
Meanwhile, the arc length of the graph $\Gamma_I$ over $I$ is given by
\[
\ell(\Gamma|_I)=\int_I \sqrt{1+f'(x)^2}\,dx.
\]
Hence,
\[
\frac{\ell(\Gamma|_I)}{\sqrt{|I|^2+(f(b)-f(a))^2}}= \frac{\frac{1}{|I|}\int_I \sqrt{1+f'(x)^2}\,dx}{\sqrt{1+\big((f')_I\big)^2}}.
\]

To estimate this quantity, we write
\begin{align*}
\left|\frac{\frac{1}{|I|}\int_I \sqrt{1+f'(x)^2}\,dx}{\sqrt{1+\big((f')_I\big)^2}} -1\right|\leq \frac{\frac{1}{|I|}\int_I \left|\sqrt{1+f'(x)^2}-\sqrt{1+\big((f')_I\big)^2}\right|\,dx}{\sqrt{1+\big((f')_I\big)^2}}.
\end{align*}
Notice that the denominator is bounded below by $1$, and the function $t\mapsto \sqrt{1+t^2}$ is $1$-Lipschitz. Therefore, $\left|\sqrt{1+x^2}-\sqrt{1+y^2}\right| \leq |x-y|$, which implies
\[
\left|\sqrt{1+f'(x)^2}-\sqrt{1+\big((f')_I\big)^2}\right|\leq \left|f'(x)-(f')_I\right|.
\]
Therefore,
\[
\left|\frac{\ell(\Gamma|_I)}{\sqrt{|I|^2+(f(b)-f(a))^2}} -1\right|\leq \frac{1}{|I|}\int_I \left|f'(x)-(f')_I\right|\,dx.
\]

Since $f'\in \mathrm{BMO}(\mathbb{R})$, the right-hand side is uniformly bounded by $\|f'\|_{\mathrm{BMO}(\mathbb{R})}$. Consequently,
\[
\frac{\ell(\Gamma|_I)}{\sqrt{|I|^2+(f(b)-f(a))^2}}\leq 1 + \|f'\|_{\mathrm{BMO}(\mathbb{R})},
\]
which proves that $\Gamma$ is a chord-arc curve. 

If $f'\in \mathrm{VMO}(\mathbb{R})$, the mean oscillation satisfies
\[
\frac{1}{|I|}\int_I \left|f'(x)-(f')_I\right|\,dx \to 0\quad \text{as } |I|\to 0,
\]
uniformly over all intervals $I$. This forces the ratio of arc length to chord length to approach $1$ as $|I|\to 0$, yielding that $\Gamma$ is asymptotically smooth.
\end{proof}

The regularity of $f'$ can also be transferred to
$\log |F'|$.

\begin{lemma}\label{lip_of_F:label}
Let $f\colon \mathbb{R}\to\mathbb{R}$ be locally absolutely continuous. Then
\[
|f'| \in X(\mathbb{R}) \implies \log |F'| \in X(\mathbb{R}),
\]
where $X=\mathrm{BMO}, \mathrm{VMO}, H^{1/2}$.
\end{lemma}

\begin{proof}
Since $f$ is locally absolutely continuous, we have $F'(x)=1+i f'(x)$ for a.e.\ $x\in\mathbb{R}$, which yields
\[
\log |F'(x)|=\frac{1}{2}\log\big(1+|f'(x)|^2\big).
\]
Thus, it suffices to investigate the composition with the function $\psi(u)=\tfrac{1}{2}\log(1+u^2)$.

Notice that $\psi$ is globally Lipschitz on $\mathbb{R}$. 
Indeed, $\psi'(u)=\frac{u}{1+u^2}\leq 1/2$ for all $u\in\mathbb R$.
In particular, for any interval $I\subset\mathbb{R}$,
\[
\frac{1}{|I|}\int_{I}\left|\psi(|f'(x)|)-\psi(|f'|_{I})\right|\, dx\leq \frac{1}{|I|}\int_{I}\big||f'(x)|-|f'|_I\big|\, dx.
\]
This shows that $\psi(|f'|)\in \mathrm{BMO}(\mathbb{R})$ (respectively, $\mathrm{VMO}(\mathbb{R})$) whenever $|f'|\in \mathrm{BMO}(\mathbb{R})$ (respectively, $\mathrm{VMO}(\mathbb{R})$).

Similarly, for the $H^{1/2}$ seminorm, the Lipschitz property again yields
\[
\frac{\left|\psi(|f'(x)|)-\psi(|f'(y)|)\right|^2}{|x-y|^2}\leq \frac{\big||f'(x)|-|f'(y)|\big|^2}{|x-y|^2}.
\]
Integrating over $\mathbb{R}\times\mathbb{R}$ establishes the desired bound for the $H^{1/2}$ seminorm.

Combining these estimates, we conclude that $\psi(|f'|)\in \mathrm{BMO}(\mathbb{R})$, $\mathrm{VMO}(\mathbb{R})$, $H^{1/2}(\mathbb{R})$ whenever $|f'|$ belongs to the corresponding space. 
\end{proof}

The following technical claim is also used.

\begin{lemma}\label{BMOimpliesA_infty}
For $u\in \mathrm{BMO}(\mathbb{R})$, $w(x)=1+|u(x)|$ is an $A_\infty$-weight on $\mathbb{R}$.
\end{lemma}

\begin{proof}
We use the following characterization of $A_\infty$ shown in \cite{Hr}: a non-negative locally integrable function $w$ belongs to $A_\infty$ if and only if
\begin{equation}\label{A_infty-condition}
\sup_I\left(\frac{1}{|I|}\int_I w(x)\,dx\right)\exp\left\{-\frac{1}{|I|}\int_I \log w(x)\,dx\right\}<\infty,
\end{equation}
where the supremum is taken over all bounded intervals $I\subset \mathbb{R}$.

Let $B=\|u\|_{{\rm BMO}(\mathbb{R})}$. For any bounded interval $I\subset \mathbb{R}$, we have
\[
\frac{1}{|I|}\int_I |u(x)|\,dx \leq|u_I|+\frac{1}{|I|}\int_I |u(x)-u_I|\,dx\leq|u_I|+B.
\]
Hence
\begin{equation}\label{integralmean}
\frac{1}{|I|}\int_I w(x)\,dx\leq 1+|u_I|+B.
\end{equation}

We next estimate the geometric mean of $w$.

If $|u_I|=0$, then \eqref{integralmean} yields
\[
\frac{1}{|I|}\int_I w(x)\,dx\leq 1+B.
\]
Since $\log w\geq 0$, we have
\[
\exp\left\{-\frac{1}{|I|}\int_I \log w(x)\,dx\right\}\leq 1.
\]
Therefore,
\[
\left(\frac{1}{|I|}\int_I w(x)\,dx\right)\exp\left\{-\frac{1}{|I|}\int_I \log w(x)\,dx\right\}\leq 1+B.
\]

Hence, it remains to consider the case $|u_I|\neq 0$. Let
\[
E=\{x\in I:\ |u(x)-u_I|\leq |u_I|/2\}.
\]
By Chebyshev's inequality,
\begin{equation}\label{cheby}
\frac{|I\setminus E|}{|I|}\leq\frac{2}{|u_I||I|}\int_I |u(x)-u_I|\,dx\leq\frac{2B}{|u_I|}.
\end{equation}
For $x\in E$, we have
\[
|u(x)|\geq |u_I|-|u(x)-u_I|\geq |u_I|/2.
\]
This gives
\begin{equation}\label{logw}
\log w(x)=\log(1+|u(x)|)\geq \log(1+|u_I|/2),\qquad x\in E.
\end{equation}
Since $\log w\geq 0$, it follows from \eqref{cheby} and \eqref{logw} that
\[
\frac{1}{|I|}\int_I \log w(x)\,dx\geq\frac{|E|}{|I|}\log(1+|u_I|/2)\geq\left(1-\frac{2B}{|u_I|}\right)\log(1+|u_I|/2).
\]
Consequently,
\begin{equation}\label{geometricmean}
\exp\left\{\frac{1}{|I|}\int_I \log w(x)\,dx\right\}\geq(1+|u_I|/2)^{1-2B/|u_I|}.
\end{equation}

The combination of \eqref{integralmean} and \eqref{geometricmean} yields
\[
\begin{aligned}
&\quad\left(\frac{1}{|I|}\int_I w(x)\,dx\right)\exp\left\{-\frac{1}{|I|}\int_I \log w(x)\,dx\right\} \\ 
&\leq\frac{1+|u_I|+B}{1+|u_I|/2}(1+|u_I|/2)^{2B/|u_I|}\leq(2+B)e^B.
\end{aligned}
\]
Since this bound is independent of $I$, we have verified \eqref{A_infty-condition}.
\end{proof}

We now combine the previous regularity results under the assumption $f'\in \mathrm{BMO}(\mathbb{R})$ to establish the strong quasisymmetry of the associated graph welding.
The condition $f'\in \mathrm{BMO}(\mathbb{R})$ is often denoted by 
$f \in \mathrm{BMO}_1(\mathbb{R})$ in the literature.

\begin{theorem}\label{bmo:label} 
Let $f\colon \mathbb{R}\to \mathbb{R}$ be a locally absolutely continuous function. 
Assume that $f'\in \mathrm{BMO}(\mathbb{R})$.
Then the associated graph welding $\varphi$ belongs to $\mathrm{SQS}(\mathbb{R})$.
\end{theorem}

\begin{proof}
Recall that $F(x)=x+if(x)$. A direct computation gives
\[
\frac{1+|f'(x)|}{2}\leq |F'(x)|=\sqrt{1+|f'(x)|^2}\leq  1+|f'(x)|.
\]
Since $1+|f'|\in A_{\infty}$ by Lemma~\ref{BMOimpliesA_infty}, it immediately follows that $|F'|\in A_{\infty}$.

Since $f$ is locally absolutely continuous, so is $F$. 
By Lemma~\ref{bmo1=ca:label}, the curve $\Gamma=F(\mathbb{R})$ is chord-arc. 
Let $\gamma$ be an arc-length parametrization of $\Gamma$. 
Then $\gamma$ is bi-Lipschitz and satisfies $|\gamma'|=1$ a.e..
Consequently, the composition $\gamma^{-1}\circ F$ is locally absolutely continuous. 
Applying the chain rule, we obtain
\[
|(\gamma^{-1}\circ F)'(x)|=|(\gamma^{-1})'\circ F(x)|\cdot |F'(x)|=|F'(x)|
\]
for almost every $x \in \mathbb{R}$. 
Since $|F'| \in A_{\infty}$, this implies that $\gamma^{-1}\circ F\in \mathrm{SQS}(\mathbb{R})$. 
Since $\mathrm{SQS}(\mathbb{R})$ forms a group under composition, its inverse $F^{-1}\circ \gamma$ also belongs to $\mathrm{SQS}(\mathbb{R})$.

On the other hand, $\Gamma$ is a chord-arc curve implies that $G|_{\mathbb{R}}$ is locally absolutely continuous and $|(G|_{\mathbb{R}})'|\in A_{\infty}$ \cite[Theorem 3.6]{WM}. 
Therefore, the composition $\gamma^{-1}\circ G|_{\mathbb{R}}$ is locally absolutely continuous, with
\[
|(\gamma^{-1}\circ G|_{\mathbb{R}})'|=|(G|_{\mathbb{R}})'|\in A_{\infty}.
\]
Hence, $\gamma^{-1}\circ G|_{\mathbb{R}}\in \mathrm{SQS}(\mathbb{R})$.

Finally, we write
\[
\varphi^{-1}=F^{-1}\circ G|_{\mathbb R},=F^{-1}\circ \gamma\circ \gamma^{-1}\circ G|_{\mathbb R}.
\]
Since both $F^{-1}\circ \gamma$ and $\gamma^{-1}\circ G|_{\mathbb{R}}$ belong to $\mathrm{SQS}(\mathbb{R})$, their composition $\varphi^{-1}$ also belongs to $\mathrm{SQS}(\mathbb{R})$. 
Therefore, $\varphi\in \mathrm{SQS}(\mathbb{R})$.
\end{proof}

The corresponding $\mathrm{VMO}$ version can be stated as follows.
As in the case of $\mathrm{S}(\mathbb{R})$ in Theorem~\ref{s:label}, we need the additional assumption $f\in L^{\infty}(\mathbb{R})$ to handle the composition of the relevant mappings.

\begin{theorem}\label{vmo:label} 
Let $f\colon \mathbb{R}\to \mathbb{R}$ be locally absolutely continuous. 
Assume that $f$ satisfies the following conditions:
\begin{enumerate}
   \item $f\in L^{\infty}(\mathbb{R})$;
   \item $f'\in \mathrm{VMO}(\mathbb{R})$.
\end{enumerate}
Then the associated graph welding $\varphi$ belongs to $\mathrm{SS}(\mathbb{R})$.
\end{theorem}

\begin{proof}
By Lemma~\ref{bmo1=ca:label}, the assumption $f'\in \mathrm{VMO}(\mathbb{R})$ implies that $\Gamma$ is an asymptotically smooth chord-arc curve. 
Applying Lemma~\ref{lip_of_F:label}, this also implies that $\log |F'|\in \mathrm{VMO}(\mathbb{R})$. 

Let $\gamma$ be an arc-length parametrization of $\Gamma$. Then $\gamma$ is bi-Lipschitz and satisfies $|\gamma'|=1$ a.e.. 
Consequently, the composition $\gamma^{-1}\circ F$ is locally absolutely continuous. 
Applying the chain rule, we obtain
\[
\log|(\gamma^{-1}\circ F)'|=\log(|(\gamma^{-1})'\circ F|\cdot |F'|)=\log |F'|\in \mathrm{VMO}(\mathbb{R}),
\]
which yields that $\gamma^{-1}\circ F\in \mathrm{SS}(\mathbb{R})$. 
Furthermore, since $\gamma$ is bi-Lipschitz and $F^{-1}$ is uniformly continuous, their composition $F^{-1}\circ \gamma$ is uniformly continuous. 
Thus $F^{-1}\circ \gamma\in \mathrm{SS}(\mathbb{R})$ by \cite[Corollary 3.2]{WM23}. 

By Lemma~\ref{strip}, $G|_{\mathbb{R}}$ is uniformly continuous, and so is $\gamma^{-1}\circ G|_{\mathbb{R}}$. 
Thus, together with the fact that $\Gamma$ is an asymptotically smooth chord-arc curve, we have $\log G'\in \mathrm{VMOA}(\mathbb{R})$ by \cite[Corollary 7.9]{MT1}. 
A direct computation yields
\[
\log |(\gamma^{-1}\circ G|_{\mathbb{R}})'|=\log |(G|_{\mathbb{R}})'|\in \mathrm{VMO}(\mathbb{R}),
\]
from which it follows that $\gamma^{-1}\circ G|_{\mathbb{R}}\in \mathrm{SS}(\mathbb{R})$.
Since $\gamma^{-1}\circ G|_{\mathbb{R}}$ is uniformly continuous and $\gamma^{-1}\circ G|_{\mathbb{R}}\in \mathrm{SS}(\mathbb{R})$, \cite[Proposition 3.1]{WM23} shows that
\[
\log|(F^{-1}\circ \gamma \circ \gamma^{-1}\circ G)'|=\log |(F^{-1}\circ \gamma)'|\circ (\gamma^{-1}\circ G)+ \log |(\gamma^{-1}\circ G)'|\in\mathrm{VMO}(\mathbb{R}). 
\]
This proves that $\varphi^{-1}=F^{-1}\circ G\in \mathrm{SS}(\mathbb{R})$. 

Finally, by Lemma~\ref{bmo-zyg}, $f'\in \mathrm{VMO}(\mathbb{R})$ implies that $f$ is a little Zygmund function. 
Consequently, $F$ is uniformly continuous. 
By Lemma~\ref{strip}, $G^{-1}|_{\Gamma}$ is also uniformly continuous, and so is $G^{-1}\circ F$. 
By \cite[Corollary 3.2]{WM23} again, it follows that $\varphi=G^{-1}\circ F\in \mathrm{SS}(\mathbb{R})$.
\end{proof}

Finally, we turn to the $H^{1/2}$ setting.
The following theorem provides a criterion for the graph welding associated 
with a locally absolutely continuous function to belong to the Weil--Petersson class.
We note that conditions (1) and (2) below are independent and they correspond to the imaginary and the real parts of $\log F'$ respectively.

\begin{theorem}\label{wp:label}
Let $f\colon \mathbb{R}\to \mathbb{R}$ be a locally absolutely continuous function. 
Assume that $f$ satisfies the following conditions:
\begin{enumerate}
   \item $\arctan f'\in H^{1/2}(\mathbb{R})$;
   \item $\log (1+|f'|^2)\in H^{1/2}(\mathbb{R})$.
\end{enumerate}
Then the associated graph welding $\varphi$ belongs to $\mathrm{WP}(\mathbb{R})$.
\end{theorem}

\begin{proof}
We begin by introducing the arc-length parametrization of the graph $\Gamma$. 
Define
\[
\sigma(t)=\int_0^t \sqrt{1+(f'(x))^2}\,dx,
\]
and let $\gamma=F\circ \sigma^{-1}$. Then $\gamma$ is an arc-length parametrization of $\Gamma$. 
By the chain rule,
\[
\gamma'=F'\circ \sigma^{-1}\cdot (\sigma^{-1})'=\left(\frac{F'}{\sigma'}\right)\circ \sigma^{-1}.
\]
Moreover, we may write the ratio as
\[
\frac{F'}{\sigma'}=\frac{1+if'}{\sqrt{1+|f'|^2}}=e^{i\theta},
\]
where $\theta=\arg \frac{F'}{\sigma'}=\arctan f'$. 
Since $\theta\in H^{1/2}(\mathbb{R})$ by assumption, it follows that $\tfrac{F'}{\sigma'}\in H^{1/2}(\mathbb{R})$.

Next, we consider the parametrization $\sigma$. 
Since
\[
\log \sigma'=\log \sqrt{1+|f'|^2}\in H^{1/2}(\mathbb{R}),
\]
\cite[Theorem 2.2]{ST20} implies that $\sigma\in \mathrm{WP}(\mathbb{R})$. 
In particular, $\sigma$ is quasisymmetric. 
By \cite[Theorem 3.1]{NS}, composition with $\sigma^{\pm1}$ preserves the space $H^{1/2}(\mathbb{R})$. 
Hence,
\[
\gamma'=\left(\frac{F'}{\sigma'}\right)\circ \sigma^{-1}\in H^{1/2}(\mathbb{R}),
\]
which shows that $\gamma\in H^{3/2}(\mathbb{R})$. Therefore, $\Gamma$ is a Weil--Petersson quasicircle by \cite[Theorem 1.1]{Bi25}.
Applying the characterization of boundary correspondences for Weil--Petersson quasicircles given in \cite[Theorem 4.4]{STW18}, we obtain
\[
\log |(G|_{\mathbb{R}})'|\in H^{1/2}(\mathbb{R}).
\]

Using the chain rule once again, we see that
\[
\log|(\gamma^{-1}\circ G|_{\mathbb{R}})'|=\log |(G|_{\mathbb{R}})'|\in H^{1/2}(\mathbb{R}),
\]
and therefore $\gamma^{-1}\circ G|_{\mathbb{R}}\in \mathrm{WP}(\mathbb{R})$.
Similarly,
\[
\log|(\gamma^{-1}\circ F)'|=\log |F'|\in H^{1/2}(\mathbb{R}),
\]
and hence $\gamma^{-1}\circ F\in \mathrm{WP}(\mathbb{R})$. 
Since $\mathrm{WP}(\mathbb{R})$ forms a group under composition (see \cite{Cui00}, \cite{TT06}), it follows that
\[
F^{-1}\circ \gamma\in \mathrm{WP}(\mathbb{R}).
\]
Finally,
\[
\varphi^{-1}=F^{-1}\circ G|_{\mathbb{R}}=F^{-1}\circ \gamma\circ \gamma^{-1}\circ G|_{\mathbb{R}}\in \mathrm{WP}(\mathbb{R}).
\]
Thus, $\varphi\in \mathrm{WP}(\mathbb{R})$.
\end{proof}

\begin{corollary}\label{last}
   Let $f\colon \mathbb{R}\to \mathbb{R}$ be a locally absolutely continuous function. 
   Assume that $f'\in H^{1/2}(\mathbb{R})$. 
   Then the associated graph welding $\varphi$ belongs to $\mathrm{WP}(\mathbb{R})$. 
\end{corollary}
\begin{proof}
   The result follows immediately from Theorem~\ref{wp:label} and the Lipschitzness of $\arctan u$ and $\log (1+u^2)$.
\end{proof}

\begin{remark}
An increasing homeomorphism $h\colon\mathbb{R}\to\mathbb{R}$ belongs to the $p$-Weil--Petersson class $\mathrm{WP}_p(\mathbb{R})$ for $p>1$ if it is locally absolutely continuous and $\log h'\in B_p(\mathbb{R})$, where the Besov space $B_p(\mathbb{R})$ consists of all functions $u\in L^1_{\mathrm{loc}}(\mathbb{R})$ such that
\[
\iint_{\mathbb{R}\times\mathbb{R}}\frac{|u(x)-u(y)|^p}{|x-y|^2}\,dx\,dy<\infty.
\]
See \cite{WM23-b}. 
In addition, we note that the arc-length parametrization $\gamma$ satisfies $\gamma' \in B_p(\mathbb{R})$ if and only if $\log \gamma' \in B_p(\mathbb{R})$.
Then, Theorem \ref{wp:label} and Corollary \ref{last} are also true by the argument in the same line after replacing $\mathrm{WP}(\mathbb{R})$ with $\mathrm{WP}_p(\mathbb{R})$ for any $p>1$ and $H^{1/2}(\mathbb{R})=B_2(\mathbb{R})$ with $B_p(\mathbb{R})$.
\end{remark}

\section{Counterparts in Conformal Weldings}\label{sec:counterparts_in_conformal_welding}

Let $\Gamma$ be the graph of a continuous function
$f\colon\mathbb{R}\to\mathbb{R}$.
Denote by $\Omega_+$ and $\Omega_-$ the domains lying above and below
$\Gamma$, respectively. Let
\[
G_+\colon\mathbb H\to\Omega_+,
\qquad
G_-\colon\mathbb H^*\to\Omega_-
\]
be conformal mappings satisfying $G_{\pm}(\infty)=\infty$.

Recall that the conformal welding $h\colon\mathbb{R} \to \mathbb{R}$ associated with $f$ (or with $\Gamma$) is defined by $h=G^{-1}_-\circ G_+|_{\mathbb{R}}$, which is an increasing self-homeomorphism of $\mathbb{R}$.

In particular, $h$ is determined uniquely by $f$ up to the pre- and post-composition with affine transformations.

As a basic property of conformal weldings associated with graph curves, we recall the following result; see \cite[Theorem 6]{AZ} and \cite[Exercise 6.6]{Pom}.

\begin{proposition}\label{qcgraph}
If $\Gamma$ is a quasicircle, then the associated conformal welding
$h$ belongs to $\mathrm{SQS}(\mathbb{R})$. In particular, $h$ is locally
absolutely continuous.
\end{proposition}

For general quasicircles, not necessarily arising as graph curves, there are many results describing how the properties of the conformal welding $h$ reflect the function space to which $\log G'$ belongs. 
Such results play an important role in the theory of the Teichm\"uller spaces discussed in the introduction.

We next consider the uniform continuity of conformal weldings.

\begin{proposition}
If the imaginary part of $\Gamma$ is bounded or equivalently $f \in L^\infty(\mathbb{R})$, then both the conformal welding $h$ and its inverse $h^{-1}$ are uniformly continuous.
\end{proposition}

\begin{proof}
By Lemma \ref{strip}, $G_+|_{\mathbb{R}}\colon \mathbb{R} \to \Gamma$ and its inverse $G_+^{-1}|_{\Gamma}$ are uniformly continuous.
We can apply this claim also to $G_-|_{\mathbb{R}}$ and $G_-^{-1}|_{\Gamma}$.
Hence, $h=(G_-)^{-1}\circ G_+|_{\mathbb{R}}$ and $h^{-1}=G^{-1}_+\circ G_-|_{\mathbb{R}}$ are uniformly continuous.
\end{proof}

For a continuous function $f\colon \mathbb{R} \to \mathbb{R}$, let $F(x)=x+if(x)$, so that $\Gamma=F(\mathbb{R})$. Define
\[
\varphi_+=G_+^{-1}\circ F,
\qquad
\varphi_-=G_-^{-1}\circ F.
\]
Then $\varphi_+$ and $\varphi_-$ are precisely the graph weldings associated
with $f$. Moreover,
\[
h=\varphi_-\circ\varphi_+^{-1},
\qquad
h^{-1}=\varphi_+\circ\varphi_-^{-1}.
\]
Using these identities, the results obtained in Section~\ref{sec:regularity_of_graph_weldings} can be translated into corresponding statements for conformal weldings.

\begin{theorem}
Let $f\colon \mathbb{R} \to \mathbb{R}$ be a continuous function. Then the conformal welding $h$ associated with $f$ satisfies the following:
\begin{enumerate}
  \item If $f$ is a Zygmund function, then $h \in \mathrm{QS}(\mathbb{R})$;
  \item If $f \in \lambda^S_*(\mathbb{R}) \cap L^\infty(\mathbb{R})$, then $h \in \mathrm{S}(\mathbb{R})$.
\end{enumerate}
Moreover, assume further that $f$ is locally absolutely continuous.
\begin{enumerate}
  \setcounter{enumi}{2}
  \item If $f' \in \mathrm{BMO}(\mathbb{R})$, then $h \in \mathrm{SQS}(\mathbb{R})$;
  \item If $f' \in \mathrm{VMO}(\mathbb{R})$ and $f \in L^\infty(\mathbb{R})$, then $h \in \mathrm{SS}(\mathbb{R})$;
  \item If $\arctan f'\in H^{1/2}(\mathbb{R})$ and
    $\log (1+|f'|^2)\in H^{1/2}(\mathbb{R})$, then $h \in \mathrm{WP}(\mathbb{R})$.
\end{enumerate}
\end{theorem}

\begin{proof}
We note that $\mathrm{QS}(\mathbb{R})$, $\mathrm{SQS}(\mathbb{R})$, and $\mathrm{WP}(\mathbb{R})$ possess a group structure, whereas $\mathrm{S}(\mathbb{R})$ and $\mathrm{SS}(\mathbb{R})$ do not.

Therefore, assertions (1), (3), and (5) follow immediately from the corresponding results for graph weldings.
To prove (2) and (4), we use the uniform continuity of graph weldings.

Indeed, the assumption $f\in L^\infty(\mathbb{R})$ implies by Lemma~\ref{strip} that both $\varphi_\pm$ and their inverses are uniformly continuous.
Hence, the compositions
\[
h=\varphi_-\circ\varphi_+^{-1},
\qquad
h^{-1}=\varphi_+\circ\varphi_-^{-1}
\]
belong to the same class by \cite[Corollaries 3.2 \& 3.4]{WM23}, yielding (2) and (4).
\end{proof}

\begin{remark}
Concerning assertion (3), Lemma~\ref{bmo-zyg} implies that $f'\in\mathrm{BMO}(\mathbb{R})$ yields a Zygmund function $f$.
By the proof of Theorem~\ref{qs:label}, the graph $\Gamma$ of $f$ is therefore a quasicircle and hence Proposition~\ref{qcgraph} gives the same conclusion that $h \in \mathrm{SQS}(\mathbb{R})$.
Alternatively, by Lemma~\ref{bmo1=ca:label}, the graph $\Gamma$ is a chord-arc curve. 
The conclusion also follows from the fact that the conformal welding of a chord-arc curve is strongly quasisymmetric.
\end{remark}

\end{document}